\documentclass[10pt,twoside]{article}

\newcommand{\AuthorName}{Ruisong Yuan}

\newcommand{\AuthorAffiliationNumbers}{1,2}
\newcommand{\AllAuthorNames}{Wenyi Wang, Ruisong Yuan, Yuehui Zhang, Yuhang Zhu}
\newcommand{\AuthorLine}{%
  Wenyi Wang\textsuperscript{1}\,\ensuremath{\cdot}\,
  \AuthorName\textsuperscript{\AuthorAffiliationNumbers,}\CorrespondingAuthorNote\,\ensuremath{\cdot}\,
  Yuehui Zhang\textsuperscript{1}\,\ensuremath{\cdot}\,
  Yuhang Zhu\textsuperscript{1}}
\newcommand{\AuthorEmail}{doubendouyrs@sjtu.edu.cn}
\newcommand{\AuthorEmailWang}{wangwenyi2004@sjtu.edu.cn}
\newcommand{\AuthorEmailZhang}{zyh@sjtu.edu.cn}
\newcommand{\AuthorEmailZhu}{12h.03@sjtu.edu.cn}

\newcommand{\AuthorDepartmentOne}{School of Mathematical Sciences}
\newcommand{\AuthorInstitutionOne}{Shanghai Jiao Tong University}
\newcommand{\AuthorStreetAddressOne}{800 Dongchuan Road}
\newcommand{\AuthorCityOne}{Shanghai}
\newcommand{\AuthorPostalCodeOne}{200240}
\newcommand{\AuthorCountryOne}{China}
\newcommand{\AuthorAffiliationOne}{%
  \AuthorDepartmentOne, \AuthorInstitutionOne, \AuthorStreetAddressOne,
  \AuthorCityOne\ \AuthorPostalCodeOne, \AuthorCountryOne}

\newcommand{\AuthorInstitutionTwo}{Shanghai Innovation Institute}
\newcommand{\AuthorStreetAddressTwo}{No. 3, Lane 699, Huafa Road}
\newcommand{\AuthorCityTwo}{Shanghai}
\newcommand{\AuthorPostalCodeTwo}{200231}
\newcommand{\AuthorCountryTwo}{China}
\newcommand{\AuthorAffiliationTwo}{%
  \AuthorInstitutionTwo, \AuthorStreetAddressTwo,
  \AuthorCityTwo\ \AuthorPostalCodeTwo, \AuthorCountryTwo}

\usepackage[T1]{fontenc}
\usepackage{mathptmx,amsmath,amssymb,amsthm}
\usepackage[scaled=0.94]{helvet}
\usepackage[paperwidth=155mm,paperheight=235mm,left=50.742bp,
  textwidth=119mm,top=55.927bp,textheight=194mm]{geometry}
\usepackage{graphicx,eso-pic,titlesec,enumitem}
\usepackage[colorlinks=true,linkcolor=blue,citecolor=blue,urlcolor=blue,
  pdftitle={Directed partial orders on the complex number field},
  pdfauthor={\AllAuthorNames}]{hyperref}
\titleformat{\section}{\sffamily\bfseries\fontsize{11}{13}\selectfont}{\thesection}{.45em}{}
\titlespacing*{\section}{0pt}{22pt plus 2pt minus 2pt}{12pt plus 1pt minus 1pt}
\titleformat{\subsection}{\sffamily\bfseries\fontsize{10}{12}\selectfont}{\thesubsection}{.45em}{}
\titlespacing*{\subsection}{0pt}{16pt plus 2pt minus 2pt}{9pt plus 1pt minus 1pt}
\newtheoremstyle{journalplain}{9pt plus 1pt minus 1pt}{9pt plus 1pt minus 1pt}
  {\itshape}{0pt}{\sffamily\bfseries}{}{.65em}{}
\newtheoremstyle{journaldefinition}{9pt plus 1pt minus 1pt}{9pt plus 1pt minus 1pt}
  {\normalfont}{0pt}{\sffamily\bfseries}{}{.65em}{}
\theoremstyle{journalplain}
\newtheorem{theorem}{Theorem}[section]
\newtheorem{lemma}[theorem]{Lemma}
\newtheorem{proposition}[theorem]{Proposition}
\newtheorem{corollary}[theorem]{Corollary}
\theoremstyle{journaldefinition}
\newtheorem*{remark}{Remark}
\newtheorem*{example}{Example}
\renewenvironment{proof}[1][Proof]{\par\pushQED{\qed}\normalfont
  \topsep6pt plus 1pt\trivlist
  \item[\hskip\labelsep{\sffamily\bfseries\itshape #1}]\ignorespaces}
  {\popQED\endtrivlist}

\numberwithin{equation}{section}
\newcommand{\C}{\mathbb C}
\newcommand{\R}{\mathbb R}
\newcommand{\Q}{\mathbb Q}
\newcommand{\Qbar}{\overline{\mathbb Q}}
\newcommand{\Z}{\mathbb Z}

\newcommand{\A}{\mathbb A}
\newcommand{\OO}{\mathcal O}

\newcommand{\ordt}{\operatorname{ord}_{t}}
\newcommand{\lc}{\operatorname{lc}}
\newcommand{\nonzero}[1]{#1\setminus\{0\}}
\newcommand{\place}[3]{\AtPageUpperLeft{\setlength{\unitlength}{1bp}\put(#1,-#2){#3}}}
\AddToShipoutPictureFG{%
  \ifnum\value{page}>1
  \place{50.742}{39.655}{\makebox[119mm][l]{\normalfont\sffamily\fontsize{8}{10}\selectfont
    \ifodd\value{page} Directed partial orders on the complex field\hfill\thepage
    \else\thepage\fi}}
  \place{50.742}{46.381}{\rule{119mm}{.7bp}}
  \fi}
\makeatletter
\renewcommand\@biblabel[1]{#1.}
\renewenvironment{thebibliography}[1]
 {\section*{References}\fontsize{8}{9.5}\selectfont
  \list{\@biblabel{\@arabic\c@enumiv}}{\setlength{\labelwidth}{10pt}
    \setlength{\leftmargin}{15.37pt}\setlength{\labelsep}{5.37pt}
    \setlength{\itemsep}{2pt}\setlength{\parsep}{0pt}
    \usecounter{enumiv}\let\p@enumiv\@empty\renewcommand\theenumiv{\@arabic\c@enumiv}}
  \sloppy\clubpenalty4000\widowpenalty4000\sfcode`\.=1000\relax}
 {\endlist}
\makeatother
\newcommand{\IntCl}{\operatorname{IntCl}_{\C}}

\newcommand{\CorrespondingAuthorMark}{%
  \begingroup\setlength{\unitlength}{1pt}\linethickness{.35pt}%
  \begin{picture}(8,5.5)
    \put(0,0){\framebox(8,5.5){}}
    \qbezier(0,5.5)(2,4.5)(4,3.5)
    \qbezier(8,5.5)(6,4.5)(4,3.5)
    \qbezier(0,0)(1.5,2)(3,4)
    \qbezier(8,0)(6.5,2)(5,4)
  \end{picture}\endgroup}
\makeatletter
\newcommand{\CorrespondingAuthorNote}{%
  \begingroup
  \renewcommand{\thefootnote}{*}%
  \long\def\@makefntext##1{\noindent\makebox[12pt][l]{\textsuperscript{*}}##1}%
  \begin{NoHyper}%
  \footnote[1]{Corresponding author.}%
  \end{NoHyper}%
  \endgroup}
\newcommand{\PrintAuthorInformation}{%
  \begingroup
  \renewcommand{\thefootnote}{}%
  \long\def\@makefntext##1{\noindent##1}%
  \begin{NoHyper}%
  \footnotetext{%
    \normalfont\fontsize{8}{9.5}\selectfont
    \setlength{\parindent}{0pt}%
    \begin{tabular}{@{}p{12pt}@{}p{\dimexpr\textwidth-12pt\relax}@{}}
      \CorrespondingAuthorMark & \AuthorName\\
      & \begingroup\urlstyle{same}\expandafter\nolinkurl\expandafter{\AuthorEmail}\endgroup\\[6pt]
      & Wenyi Wang\\
      & \begingroup\urlstyle{same}\expandafter\nolinkurl\expandafter{\AuthorEmailWang}\endgroup\\[6pt]
      & Yuehui Zhang\\
      & \begingroup\urlstyle{same}\expandafter\nolinkurl\expandafter{\AuthorEmailZhang}\endgroup\\[6pt]
      & Yuhang Zhu\\
      & \begingroup\urlstyle{same}\expandafter\nolinkurl\expandafter{\AuthorEmailZhu}\endgroup\\[6pt]
      \textsuperscript{1} & \AuthorAffiliationOne\\[6pt]
      \textsuperscript{2} & \AuthorAffiliationTwo
    \end{tabular}%
  }%
  \end{NoHyper}%
  \endgroup}
\makeatother

\begin{document}
\vspace*{27bp}
\noindent{\raggedright\sffamily\bfseries\fontsize{13}{15}\selectfont
Directed partial orders on the complex number field\par}
\vspace{17pt}
\noindent{\sffamily\bfseries \AuthorLine\par}
\PrintAuthorInformation
\vspace{28pt}
\noindent{\sffamily\bfseries Abstract}\par\noindent
We construct a class of positive cones that make $\C$ into a directed partially ordered ring. The positive cones are defined using integral closures of local rings, associated with a transcendence basis and a chosen real generator. A localization criterion also yields such orders on every transcendental extension of $\Q$. For a fixed coefficient field, we prove that two real generators define the same cone if and only if they differ by an affine map with positive real-algebraic slope and translation algebraic over that field. The nonzero positive elements are closed under inversion.
\par\vspace{12pt}
\par\vspace{12pt}
\noindent{\sffamily\bfseries Mathematics Subject Classification (2020)}\par\noindent
Primary 06F25; Secondary 12J15, 13B22

\section{Introduction}\label{sec:intro}
The Birkhoff--Pierce problem asks whether the complex field $\C$ admits a lattice order compatible with its ring operations \cite{BP}. Schwartz and Yang \cite[Corollary 4.2]{SY} proved the existence of directed partial orders on $\C$. Their argument for $\C$ uses Zorn's lemma and does not construct a concrete positive cone.

A \emph{positive cone} in a field $E$ is a subset $P$ such that
\begin{equation}\label{eq:axioms}
 P+P\subseteq P,\qquad PP\subseteq P,\qquad P\cap(-P)=\{0\}.
\end{equation}
It defines a ring order by $x\preceq y$ if $y-x\in P$. We write $x\prec y$ if $y-x\in\nonzero P$. The order is \emph{directed} if every pair has a common upper bound and a common lower bound, or equivalently, if
\[
 P-P=E.
\]
It is a \emph{lattice order} if every pair has a least upper bound and a greatest lower bound. Every lattice order is directed.

Schwartz and Yang \cite[Theorem 4.1 and Corollary 4.2]{SY} study the extension of non-Archimedean total orders on purely transcendental fields to directed partial orders on algebraic extensions. They use polynomial cones and convex ideals in the finite case, then a maximality argument based on Zorn's lemma for arbitrary algebraic extensions. They later proved that a field admits a directed partial order if and only if it has characteristic zero and is either formally real or of positive transcendence degree over $\Q$ \cite{SY23}.

Ma \cite{MaFields25} described maximal partial orders on fields algebraic over $\Q$ in terms of embeddings into $\C$. For integral domains algebraic over $\Z$, he proved that a maximal partial order is an Archimedean total order whenever it is directed \cite{MaComplex25}. In particular, if $F\subseteq\R$ is algebraic over $\Q$, then $F(i)$ admits no directed ring order. This applies to the field $\Qbar$ of all algebraic numbers. For fields of positive transcendence degree, Ma \cite[Theorem 2]{Ma25} revisited the existence theorem. His proof extends a polynomial cone to a maximal partial order and explicitly credits Dubois's work on infinite primes.

For a totally ordered subfield $F$ of $E$, an $F$-algebra order also requires $F_{\ge0}P\subseteq P$. When $F$ is non-Archimedean, Ma, Wu and Zhang \cite{MWZ17} constructed concrete cones on $F(i)$ and on the quaternion algebra over $F$. These cones arise from additive semigroups in $F_{\ge0}$; none defines a lattice order. They later classified directed $F$-algebra orders on $F(i)$, using admissible semigroups when $1>0$ \cite{MWZ18} and special convex sets when $1\not>0$ \cite{MWZ19}. Xu and Zhang \cite{XZ} unified both cases through doubly convex sets of infinitesimals and extended the classification to quaternion algebras. These classifications keep the ordered field $F$ fixed.

We construct a class of directed partial orders on $\C$ by giving their positive cones in terms of integral closures of local rings. Once a transcendence basis is fixed and a real generator is chosen, the cone is determined. We also prove a localization criterion that applies to every field of characteristic zero with positive transcendence degree over $\Q$. For a fixed coefficient field, we give a necessary and sufficient condition for two real generators to define the same cone.

All subrings in the construction contain the identity of their ambient field. Write $\A=\Qbar\cap\R$ for the field of real algebraic numbers, with its usual order. Our cones are compatible with this order. Fix a transcendence basis $B$ of $\R/\Q$ containing $\pi$, and put
\begin{equation}\label{eq:data}
 t=\pi^{-1},\qquad k=\Q(B\setminus\{\pi\}).
\end{equation}
The field $k$ consists of quotients of polynomials over $\Q$ in finitely many elements of $B\setminus\{\pi\}$, with nonzero denominators.
Let $\OO$ be the integral closure in $\C$ of the local ring
\begin{equation}\label{eq:ring}
 D=k[t]_{(t)}
 =\left\{\frac{f(t)}{g(t)}:f,g\in k[T],\ g(0)\ne0\right\}.
\end{equation}
Here $T$ is an indeterminate, and $k[t]_{(t)}$ denotes the localization of $k[t]$ at the maximal ideal $(t)$.

\begin{samepage}
\begin{theorem}[Construction]\label{thm:main}
With the notation above, the set
\begin{equation}\label{eq:cone}
 P=\{0\}\cup\bigcup_{n\in\Z}t^{-n}(\A_{>0}+t\OO)
\end{equation}
is a positive cone defining a directed ring order on $\C$. It contains $1$ and is closed under inversion of nonzero elements. For every $z\in\C$ there is an integer $N\ge1$ such that
\begin{equation}\label{eq:certificate}
 \pi^N+z,\ \pi^N-z\in\nonzero P.
\end{equation}
Moreover, $P\cap\Qbar=\A_{\ge0}$.
\end{theorem}
\end{samepage}

Every nonzero element $p$ of $P$ can be uniquely expressed as
\begin{equation}\label{eq:directform}
 p=a\pi^n+\pi^{n-1}r,\qquad n\in\Z,\quad a\in\A_{>0},\quad r\in\OO.
\end{equation}
We call $n$ the \emph{level} of $p$ and $a$ its \emph{leading coefficient}. Uniqueness is proved in Proposition~\ref{prop:leading}.

We now ask which real generators give the same positive cone. Keep $k$ fixed and put
\[
 H_k=\{b\in\R:b\text{ is algebraic over }k\},\qquad X_k=\R\setminus H_k.
\]
For a real generator $T\in X_k$, define
\begin{equation}\label{eq:rebuilt}
\begin{gathered}
 t_T=T^{-1},\qquad
 \OO_T=\IntCl\bigl(k[t_T]_{(t_T)}\bigr),\\
 P_T=\{0\}\cup\bigcup_{n\in\Z}t_T^{-n}
                  (\A_{>0}+t_T\OO_T).
\end{gathered}
\end{equation}
Here $\IntCl$ denotes integral closure in $\C$. Taking $T=\pi$ gives $t_T=t$ and $\OO_T=\OO$, so $P_\pi$ is the cone $P$ of Theorem~\ref{thm:main}.

\begin{theorem}[Equivalence condition for generators]\label{thm:affine}
For $T,U\in X_k$,
\begin{equation}\label{eq:affine}
 P_T=P_U
 \quad\Longleftrightarrow\quad
 U=aT+b\quad(a\in\A_{>0},\ b\in H_k).
\end{equation}
The coefficients $a,b$ are unique.
\end{theorem}

To prove necessity, we recover $\OO_T$ as the multiplier ring of the infinitesimals. Equality of cones gives an affine relation, and a valuation argument forces the translation term to be algebraic over $k$.

\medskip
\noindent\textit{Note on recent developments.} After the initial submission of this paper, He \cite{He26} and Bao and Zhang \cite{BaoZhang26} independently constructed lattice orders on $\C$ compatible with its ring operations, thereby answering the Birkhoff--Pierce problem in the affirmative.

\section{A construction of positive cones}\label{sec:local}
We first construct positive cones from a subring $R$ of a field $E$. Theorem~\ref{thm:main} will follow by applying this construction to the integral closure $\OO$ introduced above.

We begin with a lemma on units that will be used in the construction. Here $R^\times$ denotes the multiplicative group of units of $R$.

\begin{lemma}\label{lem:unit}
Let $R$ be a subring of a field $E$, and suppose
\begin{equation}\label{eq:local}
 0\ne t\in R,\qquad 1\notin tR,\qquad E=R[t^{-1}].
\end{equation}
Then $1+tR\subseteq R^\times$. If $F$ is any subfield of $R$, then $tR\cap F=\{0\}$.
\end{lemma}
\begin{proof}
Let $w=1+tx$ with $x\in R$. Properness of $tR$ gives $w\ne0$. Write $w^{-1}=t^{-m}y$ with $m\ge0$ and $y\in R$, so that $wy=t^m$. If $m\ge1$, then
\[
 y=t^m-txy=t(t^{m-1}-xy).
\]
Thus $y=ty_1$ with $y_1\in R$, and cancellation gives $wy_1=t^{m-1}$. Repeating this step yields an inverse of $w$ in $R$. When $m=0$, the conclusion is immediate.

If $0\ne a\in tR\cap F$, then $a^{-1}\in F\subseteq R$, so $1=a^{-1}a\in tR$, a contradiction.
\end{proof}

\begin{samepage}
The following theorem gives the construction used to prove Theorem~\ref{thm:main}. It establishes the cone axioms, directedness and closure under inversion of nonzero positive elements.

\begin{theorem}[Localization criterion]\label{thm:local}
Suppose \eqref{eq:local} holds and $R$ contains a totally ordered subfield $F$. Then
\begin{equation}\label{eq:generalcone}
 P_{R}=\{0\}\cup\bigcup_{n\in\Z}t^{-n}(F_{>0}+tR)
\end{equation}
is the positive cone of a directed ring order on $E$. It contains $1$ and satisfies
\begin{equation}\label{eq:inverse}
 (\nonzero {P_{R}})^{-1}=\nonzero {P_{R}}.
\end{equation}
If $t^mz\in R$ for some integer $m\ge0$, then $N=m+1$ satisfies
\begin{equation}\label{eq:bounds}
 t^{-N}+z,\ t^{-N}-z\in\nonzero {P_{R}}.
\end{equation}
\end{theorem}
\end{samepage}
\begin{proof}
Lemma~\ref{lem:unit} shows that $a+tr\ne0$ for $a\in F_{>0}$ and $r\in R$. Let
\[
 x=t^{-n}(a+tr),\qquad y=t^{-m}(b+ts),
\]
where $a,b\in F_{>0}$ and $r,s\in R$. Assume $n\ge m$. If $n=m$, then
\begin{equation}\label{eq:addsame}
 x+y=t^{-n}\bigl((a+b)+t(r+s)\bigr).
\end{equation}
If $n>m$, then
\begin{equation}\label{eq:adddifferent}
 x+y=t^{-n}\bigl(a+t[r+t^{n-m-1}(b+ts)]\bigr).
\end{equation}
In both cases $x+y\in\nonzero{P_{R}}$, proving closure under addition and \mbox{$P_{R}\cap(-P_{R})=\{0\}$}.

Closure under multiplication follows from
\begin{equation}\label{eq:product}
 xy=t^{-(n+m)}\bigl(ab+t(as+br+trs)\bigr)\in\nonzero{P_{R}}.
\end{equation}
Also $1\in P_{R}$.

For $z\in E$, localization gives $r=t^mz\in R$ for some $m\ge0$. With $N=m+1$,
\begin{equation}\label{eq:localcertificate}
 t^{-N}\pm z=t^{-N}(1\pm tr)\in\nonzero{P_{R}}.
\end{equation}
Since $t^{-N}\in P_{R}$, the identity $z=(t^{-N}+z)-t^{-N}$ proves directedness.

Finally, put $v=r/a\in R$. Lemma~\ref{lem:unit} gives $u=(1+tv)^{-1}\in R$, and
\[
 (a+tr)^{-1}=a^{-1}u=a^{-1}-ta^{-1}vu\in a^{-1}+tR.
\]
Hence $x^{-1}\in t^n(F_{>0}+tR)$. Applying inversion twice proves \eqref{eq:inverse}.
\end{proof}

\begin{remark}
The ring $R$ need not be a valuation ring, and $R/tR$ need not be a field. The set $F_{>0}+tR$ in \eqref{eq:generalcone} consists of the elements $x\in R$ such that $x-a\in tR$ for some $a\in F_{>0}$. Equivalently, $x$ and $a$ have the same residue modulo $tR$. This $a$ is unique because $F\cap tR=\{0\}$ by Lemma~\ref{lem:unit}.
\end{remark}

\begin{proposition}[Unique leading coefficients]\label{prop:leading}
For $p\in\nonzero{P_{R}}$, write $p=t^{-n}(a+tr)$ with $n\in\Z$, $a\in F_{>0}$ and $r\in R$. Then $n$, $a$ and $r$ are uniquely determined by $p$. Call $n$ the \emph{level} $\ell(p)$ and $a$ the \emph{leading coefficient} $\lc(p)$. For nonzero positive $p,q$,
\begin{align}
 \ell(pq)&=\ell(p)+\ell(q),&\lc(pq)&=\lc(p)\lc(q),\label{eq:leadingproduct}\\
 \ell(p+q)&=\max\{\ell(p),\ell(q)\}.&&\label{eq:leadingsum}
\end{align}
If the levels agree, the leading coefficients add. Otherwise, the leading coefficient is that of the summand with larger level. Also
\begin{equation}\label{eq:leadinginverse}
 \ell(p^{-1})=-\ell(p),\qquad \lc(p^{-1})=\lc(p)^{-1}.
\end{equation}
\end{proposition}
\begin{proof}
Suppose $t^{-n}(a+tr)=t^{-m}(b+ts)$. If $n>m$, then $a+tr=t^{n-m}(b+ts)\in tR$, contrary to $tR\cap F=\{0\}$. The case $m>n$ is identical. Thus $n=m$, and $a-b\in tR\cap F$ gives $a=b$. Cancellation then gives $r=s$. The formulas follow from \eqref{eq:addsame}--\eqref{eq:product} and the inverse calculation above.
\end{proof}

The positive units are described by $G_t=F_{>0}+tR$. The product and inverse calculations show that $G_t$ is a multiplicative subgroup of $R^\times$, and
\begin{equation}\label{eq:units}
 P_{R}\cap R^\times=G_t,\qquad
 P_{R}\cap R\subseteq F_{\ge0}+tR.
\end{equation}
Indeed, an element of positive level cannot lie in $R$. An element of negative level lies in $tR$, so it is not a unit. An element of level zero belongs to $G_t$.

\subsection{Application to the complex field}\label{sec:integral}
Theorem~\ref{thm:main} is a corollary of Theorem~\ref{thm:local}. We now verify that the ring and element used to construct the positive cone on $\C$ satisfy its hypotheses, taking $R=\OO$, $E=\C$ and $F=\A$.

Both $\pi$ and $t=\pi^{-1}$ are transcendental over $k$. Indeed, suppose that $\pi$ satisfies a nonzero polynomial with coefficients in $k$. All its coefficients are rational functions over $\Q$ in some finite subset of $B\setminus\{\pi\}$. Multiplying by a common nonzero denominator gives a nonzero polynomial relation over $\Q$ among $\pi$ and those elements of $B$, contradicting the algebraic independence of $B$. If $t$ were algebraic over $k$, then its inverse $\pi$ would also be algebraic over $k$. Moreover,
\begin{equation}\label{eq:algebraicity}
 k(t)=\Q(B),\qquad \R/k(t)\text{ is algebraic},\qquad
 \C/k(t)\text{ is algebraic}.
\end{equation}
The last statement follows from $\C=\R(i)$.

Substitution $X\mapsto t$ identifies $k(X)$ with $k(t)$. Therefore
\begin{equation}\label{eq:residue}
 \varepsilon:D\longrightarrow k,\qquad
 \varepsilon\!\left(\frac{f(t)}{g(t)}\right)=\frac{f(0)}{g(0)}
\end{equation}
is well-defined and has kernel $tD$. An element of $D$ with nonzero residue modulo $tD$ is a unit. Thus $D$ is local with maximal ideal $tD$.

For $a\in k(t)\setminus\{0\}$, write
\[
 a=t^v\frac{f(t)}{g(t)},\qquad f(0)g(0)\ne0,
\]
and define $\ordt(a)=v\in\Z$. Thus $a\in D$ exactly when $v\ge0$. Set $\ordt(0)=+\infty$.

The integral elements form a $D$-subalgebra $\OO$ of $\C$ \cite[Lemma 10.36.7]{StacksI}. Every $\alpha\in\Qbar$ satisfies a monic polynomial over $\Q\subseteq D$, so $\Qbar\subseteq\OO$. We also have $1\notin t\OO$. Indeed, if $t^{-1}$ were integral over $D$, a monic relation
\[
 t^{-d}+c_1t^{-(d-1)}+\cdots+c_d=0,\qquad c_j\in D,
\]
would give $1=0$ after multiplication by $t^d$ and application of $\varepsilon$. It remains to check $\C=\OO[t^{-1}]$.

\begin{lemma}[Clearing poles]\label{lem:scaling}
For every $z\in\C$, some integer $m\ge0$ satisfies $t^mz\in\OO$. More explicitly, suppose
\begin{equation}\label{eq:polynomial}
 z^d+a_1(t)z^{d-1}+\cdots+a_d(t)=0,\qquad a_j(t)\in k(t).
\end{equation}
Set $s_j=\max\{0,-\ordt(a_j)\}$ when $a_j\ne0$, and $s_j=0$ otherwise. Any integer
\begin{equation}\label{eq:exponent}
 m\ge\max_{1\le j\le d}\left\lceil\frac{s_j}{j}\right\rceil
\end{equation}
satisfies $t^mz\in\OO$. In particular, $\OO[t^{-1}]=\C$.
\end{lemma}
\begin{proof}
A monic relation \eqref{eq:polynomial} exists by \eqref{eq:algebraicity}. Condition \eqref{eq:exponent} implies $mj+\ordt(a_j)\ge0$ for every nonzero coefficient, so $t^{mj}a_j\in D$. For $w=t^mz$, multiplication of \eqref{eq:polynomial} by $t^{md}$ gives
\begin{equation}\label{eq:scaledpolynomial}
 w^d+t^ma_1w^{d-1}+t^{2m}a_2w^{d-2}+\cdots+t^{md}a_d=0.
\end{equation}
This equation is monic over $D$, proving $w\in\OO$. Hence $z=t^{-m}w\in\OO[t^{-1}]$. The reverse inclusion is immediate.
\end{proof}

The preceding verification and Lemma~\ref{lem:scaling} establish \eqref{eq:local} with $R=\OO$ and $E=\C$. Since $\Qbar\subseteq\OO$, Lemma~\ref{lem:unit} also gives $t\OO\cap\Qbar=\{0\}$. Thus
\begin{equation}\label{eq:proper}
 1\notin t\OO,\qquad t\OO\cap\Qbar=\{0\}.
\end{equation}

\begin{corollary}\label{cor:complex}
For every $T\in X_k$, the cone $P_T$ defines a directed ring order on $\C$. It contains $1$, is closed under inversion of nonzero elements, and satisfies $P_T\cap\Qbar=\A_{\ge0}$. For each $z\in\C$, there is an integer $N\ge1$ such that $T^N+z,T^N-z\in\nonzero{P_T}$.
\end{corollary}
\begin{proof}
Since $\operatorname{trdeg}_k\R=1$, the field $\C$ is algebraic over $k(T)$ for every $T\in X_k$. The preceding verification therefore applies with $t=t_T$ and $\OO=\OO_T$. In particular, $\OO_T[t_T^{-1}]=\C$ and $t_T\OO_T\cap\Qbar=\{0\}$. Theorem~\ref{thm:local} now gives the cone axioms, directedness and closure under inversion. If $r=t_T^mz\in\OO_T$ and $N=m+1$, then
\begin{equation}\label{eq:explicitcertificate}
 T^N\pm z=t_T^{-N}(1\pm t_Tr)\in\nonzero{P_T}.
\end{equation}
Every nonzero $\alpha\in\Qbar$ is a unit of $\OO_T$, since $\alpha^{-1}\in\Qbar\subseteq\OO_T$. If also $\alpha\in P_T$, \eqref{eq:units} gives $\alpha=a+t_Tr$ with $a\in\A_{>0}$. Then $\alpha-a\in t_T\OO_T\cap\Qbar=\{0\}$, so $\alpha=a$. The reverse inclusion $\A_{\ge0}\subseteq P_T$ holds by definition.
\end{proof}

Taking $T=\pi$ proves Theorem~\ref{thm:main}.\qed

With $\Q$ as the ordered coefficient field, the same construction applies to every field covered by \cite[Corollary 4.2]{SY}.

\begin{corollary}\label{cor:generalfields}
Let $E$ be a field of characteristic zero with $\operatorname{trdeg}_{\Q}E>0$. Choose a transcendence basis $\mathcal T$ of $E/\Q$ and an element $\theta\in\mathcal T$. Put
\[
 k_0=\Q(\mathcal T\setminus\{\theta\}),\qquad s=\theta^{-1},
\]
and let $R$ be the integral closure of $k_0[s]_{(s)}$ in $E$. Then
\begin{equation}\label{eq:fieldcone}
 P_E=\{0\}\cup\bigcup_{n\in\Z}s^{-n}(\Q_{>0}+sR)
\end{equation}
is the positive cone of a directed ring order on $E$. It contains $1$ and is closed under inversion of nonzero elements.
\end{corollary}
\begin{proof}
The field $E$ is algebraic over $k_0(s)$. The clearing-poles argument of Lemma~\ref{lem:scaling} gives $E=R[s^{-1}]$. Evaluation at $s=0$ shows that $s^{-1}$ cannot be integral over $k_0[s]_{(s)}$: a monic relation would give $1=0$ after multiplication by a power of $s$. Thus $1\notin sR$. Since $\Q\subseteq R$, Theorem~\ref{thm:local} applies with $t=s$ and $F=\Q$.
\end{proof}

\begin{example}[The elements $i$ and $\sqrt\pi$]
Since $i\in\Qbar\subseteq\OO$, we have
\[
 \pi+i=t^{-1}(1+ti)\in\nonzero P,\qquad
 \pi-i=t^{-1}(1-ti)\in\nonzero P.
\]
Thus $-\pi\prec i\prec\pi$, and $i=(\pi+i)-\pi$ is a difference of two positive elements. Nevertheless, $i$ is incomparable with zero, since neither $i$ nor $-i$ belongs to $P\cap\Qbar=\A_{\ge0}$.

Now let $z=\sqrt\pi$ be the usual positive real square root. The equation $z^2-t^{-1}=0$ allows $m=1$ in Lemma~\ref{lem:scaling}. The element $r=tz$ satisfies $r^2-t=0$, so $r\in\OO$. Hence
\[
 \pi^2\pm\sqrt\pi=t^{-2}(1\pm tr)\in\nonzero P.
\]
In fact, $\sqrt\pi$ is also incomparable with zero. If $p=\sqrt\pi$ or $p=-\sqrt\pi$ were positive, then $p^2=\pi$ would give $2\ell(p)=\ell(\pi)=1$, which is impossible for $\ell(p)\in\Z$. The order on real elements therefore differs from the usual order.
\end{example}

\begin{proposition}\label{prop:nonlattice}
The elements $0$ and $i$ have no least upper bound in any of the orders defined by $P_T$, $T\in X_k$.
\end{proposition}
\begin{proof}
Fix $T$ and write $t=t_T$, $\OO=\OO_T$, and $P=P_T$.
Let $u$ be a common upper bound of $0$ and $i$. Then $u$ and $u-i$ are nonzero positive elements. We first show that $u\notin\OO$. If $u\in\OO$, then $u-i\in\OO$ as well. By \eqref{eq:units}, there are $a,b\in\A_{\ge0}$ such that
\[
 u\equiv a\pmod{t\OO},\qquad u-i\equiv b\pmod{t\OO}.
\]
Thus $i-(a-b)\in t\OO\cap\Qbar=\{0\}$, a contradiction.

Write $u=t^{-n}(c+tr)$ with $c\in\A_{>0}$ and $r\in\OO$. Since $u\notin\OO$, we have $n\ge1$. Therefore
\[
 \frac{u}{2}-i
 =t^{-n}\left(\frac c2+t\left[\frac r2-t^{n-1}i\right]\right)
 \in\nonzero P.
\]
Also $u/2\in\nonzero P$. Hence $u/2$ is a common upper bound of $0,i$, and $u/2\prec u$. No common upper bound can be least.
\end{proof}

\section{Which generators induce the same cone?}\label{sec:equivalence}
To determine when two constructions give the same cone, we first recover the ring $R$ from the cone. We show that the infinitesimals relative to $\A$ form the ideal $tR$. The ring $R$ then consists exactly of those $x\in\C$ for which $x(tR)\subseteq tR$.

Call $(R,t)$ \emph{admissible} if $R$ is a subring of $\C$ containing $\A$ and \eqref{eq:local} holds with $E=\C$. Write
\[
 P(R,t)=\{0\}\cup\bigcup_{n\in\Z}t^{-n}(\A_{>0}+tR)
\]
and define
\begin{equation}\label{eq:intrinsic}
 I(P)=\{z\in\C:\varepsilon\pm z\in P\text{ for every }\varepsilon\in\A_{>0}\},
 \qquad (I:I)=\{x\in\C:xI\subseteq I\}.
\end{equation}

\begin{theorem}[Ring reconstruction and equality of cones]\label{thm:recovery}
For every admissible pair $(R,t)$,
\begin{equation}\label{eq:recovery}
 I(P(R,t))=tR,\qquad (I(P(R,t)):I(P(R,t)))=R.
\end{equation}
For admissible $(R,t)$ and $(R',s)$,
\begin{equation}\label{eq:equality}
 P(R,t)=P(R',s)
 \quad\Longleftrightarrow\quad R=R'\text{ and }t/s\in\A_{>0}+tR.
\end{equation}
Writing $T=t^{-1}$ and $U=s^{-1}$, the condition $t/s\in\A_{>0}+tR$ becomes $U=aT+r$ with $a\in\A_{>0}$ and $r\in R$.
\end{theorem}
\begin{proof}
The inclusion $tR\subseteq I(P(R,t))$ follows from the definition. Conversely, let $z\in I(P(R,t))$. The elements $1+z$ and $1-z$ are nonnegative and sum to $2$. By \eqref{eq:leadingsum}, every nonzero summand has level at most zero. Thus both lie in $R$, and $z\in R$. Apply \eqref{eq:units} to $1+z$ and subtract $1$ to obtain $z\equiv b\pmod{tR}$ for some $b\in\A$. For every $\varepsilon\in\A_{>0}$, the residues of $\varepsilon\pm z$ give $\varepsilon\pm b\ge0$. If $b\ne0$, take $\varepsilon=|b|/2$ to get a contradiction. Hence $z\in tR$. Finally, cancellation gives $(tR:tR)=(R:R)=R$.

Equality of cones now yields $R=R'$ and $tR=sR$. Consequently $t/s$ is a unit of $R$, and it is positive because $t$ and $s^{-1}$ are positive. Formula \eqref{eq:units} gives $t/s\in G_t=\A_{>0}+tR$. Conversely, if $u=t/s\in G_t$ and $R=R'$, then $sR=tR$ and the group property of $G_t$ gives
\[
 s^{-n}(\A_{>0}+sR)=t^{-n}u^nG_t=t^{-n}G_t\qquad(n\in\Z).
\]
Taking unions proves equality. Multiplying $t/s=a+tr$ by $t^{-1}$ gives $U=aT+r$.
\end{proof}

For fixed $R$, the generators of the same cone form $\A_{>0}T+R$. Indeed, if $U=aT+r$, then $U^{-1}=t(a+tr)^{-1}$, so $U^{-1}R=tR$ and $R[U]=R[T]=\C$. For the family \eqref{eq:rebuilt}, the ring itself depends on the generator. We must therefore compare $\OO_T$ and $\OO_U$ as well.

\begin{proof}[Proof of Theorem~\ref{thm:affine}]
Suppose $P_T=P_U$. Theorem~\ref{thm:recovery} gives $\OO_T=\OO_U$ and $U=aT+b$, where $a\in\A_{>0}$ and $b\in\OO_T\cap\OO_U$. Since $T,U,a$ are real, so is $b$. It remains to prove that $b$ is algebraic over $k$.

Suppose otherwise. Then $b$ is transcendental over $H=k(a)$. We choose a valuation at which $b$ has a pole, and use it to test membership in both integral closures. More precisely, let $V$ be a valuation ring of $\C$ dominating $H[b^{-1}]_{(b^{-1})}$. Such a ring exists, is integrally closed, and satisfies
\[
 x\notin V\ \Longrightarrow\ x^{-1}\in\mathfrak m_V
\]
by \cite[Lemmas 10.50.2--10.50.4]{StacksV}. We have $H\subseteq V$ and $b^{-1}\in\mathfrak m_V$, hence $b\notin V$.

If $T\notin V$, then $T^{-1}\in\mathfrak m_V$. For $g\in k[X]$ with $g(0)\ne0$,
\[
 g(T^{-1})\equiv g(0)\not\equiv0\pmod{\mathfrak m_V},
\]
so $g(T^{-1})$ is a unit of $V$. Consequently $k[T^{-1}]_{(T^{-1})}\subseteq V$, and integral closedness gives $\OO_T\subseteq V$. This contradicts $b\notin V$. Thus $T\in V$. The same argument gives $U\in V$, but then $b=U-aT\in V$, again a contradiction. Therefore $b\in H_k$.

Conversely, let $U=aT+b$ as in \eqref{eq:affine}, and put $H=k(a,b)$, $t=T^{-1}$, $s=U^{-1}$. We first show that extending the constants from $k$ to $H$ leaves this integral closure unchanged:
\begin{equation}\label{eq:constants}
 \IntCl(k[t]_{(t)})=\IntCl(H[t]_{(t)}).
\end{equation}
Indeed, $H\subseteq\OO_T$ since $H/k$ is algebraic. If $g(0)\ne0$ for $g\in H[X]$, write $g(t)=g(0)(1+tv)$ with $v\in H[t]$. Lemma~\ref{lem:unit} gives $g(t)^{-1}\in\OO_T$. Thus
$k[t]_{(t)}\subseteq H[t]_{(t)}\subseteq\OO_T$;
the middle ring is integral over the first, and transitivity of integrality proves \eqref{eq:constants} \cite[Lemma 10.36.6]{StacksI}. The same reasoning applies to $s$.

The substitutions
\[
 s=\frac{t}{a+bt},\qquad t=\frac{as}{1-bs}
\]
give $s\in tH[t]_{(t)}$ and $t\in sH[s]_{(s)}$. A polynomial in $s$ with nonzero constant term is therefore a unit of $H[t]_{(t)}$, and the converse holds with $s,t$ interchanged. Hence $H[t]_{(t)}=H[s]_{(s)}$. Equation \eqref{eq:constants} gives $\OO_T=\OO_U$, while $t/s=a+bt\in\A_{>0}+t\OO_T$. Theorem~\ref{thm:recovery} yields equality of the cones.

If $aT+b=a'T+b'$ and $a\ne a'$, then $T=(b'-b)/(a-a')\in H_k$, a contradiction. Thus $a=a'$ and $b=b'$.
\end{proof}

Since $H_k$ is a field containing $\A$, the maps $T\mapsto aT+b$ in \eqref{eq:affine} form a group acting on $X_k$. Their orbits
\begin{equation}\label{eq:classes}
 [T]_k=\A_{>0}T+H_k
\end{equation}
are in bijection with $\mathcal C_k=\{P_T:T\in X_k\}$ through $[T]_k\mapsto P_T$. For example, $\pi$, $2\pi+1$ and $\sqrt2\,\pi+\sqrt3$ give the same cone. The generator $\pi^2$ gives a different cone, since an affine relation would make $\pi$ algebraic over $k$.

\section{Concluding remarks}\label{sec:conclusion}
He \cite{He26} and Bao and Zhang \cite{BaoZhang26} have independently constructed compatible lattice orders on $\C$; unlike their positive cones, ours are closed under inversion of every nonzero element.

Theorem~\ref{thm:affine} characterizes when two real generators define the same cone for a fixed coefficient field $k$. More generally, one may ask when two pairs $(k,T)$ and $(k',U)$, with possibly different coefficient fields, define the same cone. Theorem~\ref{thm:recovery} gives a criterion involving the corresponding integral closures. A natural further problem is to express this criterion directly in terms of the coefficient fields and generators.

\section*{Acknowledgements}
R. Yuan (the corresponding author) is partially supported by the ``Bole Plan'' of Shanghai Jiao Tong University.

\end{document}